\documentclass[12pt,a4paper,twoside]{amsart}

\numberwithin{equation}{section}
\usepackage{rotating}
\usepackage{mathrsfs}
\usepackage{amscd,amssymb,amsopn,amsmath,amsthm,graphics,amsfonts,enumerate,verbatim,calc}
\usepackage[all]{xy}
\usepackage[utf8]{inputenc}
\usepackage{color}

\newtheorem{theorem}{Theorem}[section]
\newtheorem{claim}[theorem]{Claim}

\newtheorem{lemma}[theorem]{Lemma}

\theoremstyle{definition}
\newtheorem{definition}[theorem]{Definition}

\newtheorem{fact}[theorem]{Fact}
\newtheorem{discussion}[theorem]{Discussion}

\theoremstyle{remark}
\newtheorem{remark}[theorem]{Remark}
\newtheorem{notation}[theorem]{Notation}

\newcommand{\lh}{{\ell g}}
\newcommand{\rest}{{\restriction}}

\newcommand{\Ker}{{\rm Ker}}

\newcommand{\otp}{{\rm otp}}

\newcommand{\pure}{{\rm pure}}

\newcommand{\id}{{\rm id}}

\newcommand{\Ext}{{\rm Ext}}

\newcommand{\cf}{{\rm cf}}
\newcommand{\Hom}{{\rm Hom}}

\newcommand{\Rang}{\rm{Im}}

\newcommand{\mn}{{\medskip\noindent}}
\newcommand{\sn}{{\smallskip\noindent}}

\newcommand{\bbZ}{{\mathbb Z}}

\newcommand{\cT}{{\mathcal T}}

\newcommand{\varp}{{\varepsilon}}

\newcommand{\pr}{{\rm pr}}

\newcount\skewfactor
\def\mathunderaccent#1#2 {\let\theaccent#1\skewfactor#2
\mathpalette\putaccentunder}
\def\putaccentunder#1#2{\oalign{$#1#2$\crcr\hidewidth
\vbox to.2ex{\hbox{$#1\skew\skewfactor\theaccent{}$}\vss}\hidewidth}}

\newenvironment{PROOF}[2][\proofname.]
   {\begin{proof}[#1]\renewcommand{\qedsymbol}{\textsquare$_{\rm #2}$}}
   {\end{proof}}

\usepackage{hyperref}

\begin{document}

\title {Failure of singular compactness for $\Hom$}
\author[M. Asgharzadeh]{Mohsen Asgharzadeh}

\address{Mohsen Asgharzadeh, Hakimiyeh, Tehran, Iran.}

\email{mohsenasgharzadeh@gmail.com}

\author[M.  Golshani]{Mohammad Golshani}

\address{Mohammad Golshani, School of Mathematics, Institute for Research in Fundamental Sciences (IPM), P.O.\ Box:
	19395--5746, Tehran, Iran.}

\email{golshani.m@gmail.com}

\author[S. Shelah]{Saharon Shelah}

\address{Saharon Shelah, Einstein Institute of Mathematics, The Hebrew University of Jerusalem, Jerusalem,
	91904, Israel, and Department of Mathematics, Rutgers University, New Brunswick, NJ
	08854, USA.}

\email{shelah@math.huji.ac.il}
\thanks{The second author's research has been supported by a grant from IPM (No. 1403030417). The third author research partially supported by the Israel Science Foundation
(ISF) grant no: 1838/19, and Israel Science Foundation (ISF) grant no: 2320/23; Research partially
supported by the grant “Independent Theories” NSF-BSF, (BSF 3013005232). The third author is grateful to Craig Falls for providing typing services that were used during the work on the paper. This is publication
number 1264 of the third author.}





\begin{abstract}
	Assuming G\"odel's axiom of constructibility $\mathbf{V}=\mathbf{L}$,  we  construct a $\chi$-free abelian group
	$G$ of singular cardinality  for some suitable   cardinal $\chi$ which is regular and uncountable, equipped with the property that for every nontrivial subgroup $G' \subseteq G$ of smaller cardinality,
	$\Hom(G',\mathbb{Z}) \neq 0$, while $\Hom(G,\mathbb{Z}) = 0$. This provides a consistent counterexample to the singular compactness
	of nontrivial duality with respect to the functor $\Hom(-,\mathbb{Z})$.
\end{abstract}

\maketitle
\numberwithin{equation}{section}
\setcounter{section}{-1}

\section{introduction}
Hill \cite{HIL} proved that if an abelian group \(G\) has  cardinality a singular cardinal with cofinality at most \( \omega_{1} \), and every subgroup of \(G\)  of smaller cardinality is free, then \(G\) itself is free. This result forms one of the foundations of Shelah's \emph{Singular Compactness Theorem} \cite{Sh:52}, in which he introduced an abstract notion of freeness and removed the restriction on cofinality. In particular, Shelah showed that if an abelian group \(G\) is of size a singular  cardinal  $\lambda$, and if every subgroup of \(G\) of cardinality \(<\lambda\) is free, then \(G\) must be free. General references on singular compactness may be found in \cite{EM02, fuchs}, while applications are discussed in \cite{GT}.

Compactness principles, and the corresponding phenomena of incompactness, occupy a central place in modern set-theoretic algebra. Broadly, a compactness principle asserts that if every small subobject of a given object satisfies a property \( \Pr \), then the object itself satisfies \( \Pr \). At the ASL Annual Meeting in Irvine (2008), Shelah \cite{sh:n} announced several theorems concerning singular incompactness, although these results were not published.
The aim of this paper is to provide a systematic analysis of incompactness for abelian groups with respect to the duality functor
\(
\Hom(-,\mathbb{Z}),
\)
focusing on singular cardinals. In particular, we verify a prediction from \cite{sh:n}. The property of interest is:
\begin{enumerate}
	\item[\( \Pr_{\lambda} \):] Let \(G\) be a group of cardinality \( \lambda \). If every nontrivial subgroup \( G' \subseteq G \) of cardinality \( <\lambda \) satisfies \( \Hom(G',\mathbb{Z}) \neq 0 \), then \( \Hom(G,\mathbb{Z}) \neq 0 \).
\end{enumerate}

For \( \mu \leq \lambda \) we recall that
\(
S^{\lambda}_{\mu} = \{ \alpha < \lambda : \cf(\alpha) = \mu \}
\)
is stationary in \( \lambda \). Given a stationary set \(S \subseteq \lambda\), we write \( \diamondsuit_{S} \) for Jensen's diamond on \(S\) (see Definition~\ref{dia}). When \( \lambda > \aleph_{0} \) is regular and \( \diamondsuit_{S} \) holds for some stationary non-reflecting \( S \subseteq S^{\lambda}_{\aleph_{0}} \), it is known (see~{Fact}~\ref{ej}, recorded from  \cite{EM77}) that one may construct a \( \lambda \)-free abelian group \(G\) of cardinality \( \lambda \) such that
\(
\Hom(G,\mathbb{Z}) = 0.
\)
Any subgroup \(G'\) of cardinality \(<\lambda\) is then free, and hence
\(
\Hom(G',\mathbb{Z}) \neq 0.
\)
Thus \( \Pr_{\lambda} \) fails for such \(\lambda\). These assumptions entail that \( \lambda \) is not weakly compact, since weak compactness implies reflection of stationary sets, whereas we require a non-reflecting stationary subset. In fact, the non–weak compactness of \(\lambda\) is necessary for the property \( \Pr_{\lambda} \); see Fact~\ref{weapr}. However, the argument above does not extend to singular cardinals.

Our main contribution is to show that the failure of $\Pr_\lambda$ at singular
cardinals is consistent with G\"odel’s axiom of constructibility. Although the
main theorem is established in a more general setting, the following special
case suffices for the purposes of the introduction.

\begin{theorem}\label{mt}
	Assume $\mathbf{V}=\mathbf{L}$ and that:
	\begin{enumerate}
		\item[(a)] $\langle \lambda_i : i < \kappa \rangle$ is an increasing sequence of
		regular cardinals with limit $\lambda$;
		\item[(b)] each $\lambda_i=\mu_i^{+}$ is a successor cardinal, where $\mu_i$ is
		regular and not weakly compact;
		\item[(c)] $\aleph_{0} < \kappa \leq \chi < \lambda_{0}$, where $\kappa$ and
		$\chi$ are regular.
	\end{enumerate}
	Then there exists a $\chi$-free abelian group $G$ of cardinality $\lambda$
	which is a counterexample to singular compactness at $\lambda$ for the
	property $\Hom(-,\mathbb Z)\neq 0$.
\end{theorem}
\medskip
\noindent
In fact, the assumptions of Theorem~\ref{mt} can be weakened as follows:

\begin{enumerate}
	\item[(a)] $\langle \lambda_i : i < \kappa \rangle$ is an increasing sequence of
	regular cardinals with limit $\lambda$, with each $\lambda_i=\mu_i^{+}$;
	\item[(b)] $\aleph_{0} < \kappa \leq \chi < \lambda_{0}$, where $\kappa$ and
	$\chi$ are regular;
	\item[(c)] for each $i < \kappa$, there exist stationary, non-reflecting sets
	$S_i \subseteq S^{\lambda_i}_{\aleph_{0}}$ and
	$T_i \subseteq S^{\mu_i}_{\aleph_{0}}$ such that $\diamondsuit_{S_i}$ and
	$\diamondsuit_{T_i}$ hold;
	\item[(d)] there is no measurable cardinal $\leq \lambda$.
\end{enumerate}
\medskip

Our investigation is closely related to the classical Whitehead problem, concerning the vanishing of
\(\Ext(G,\mathbb{Z}):=\Ext_{\mathbb{Z}}^1(G,\mathbb{Z})
,
\)
a property of considerable interest but greater complexity. In our forthcoming work \cite{ags2}, we study singular compactness in the context of \( \Ext \). Under the additional assumption \( \mathbf{V} = \mathbf{L} \), the situation becomes simpler: by \cite{Sh:44}, an abelian group \(G\) of cardinality \( \lambda > \aleph_{0} \) is free if and only if
\(
\Ext(G,\mathbb{Z}) = 0.
\)
Hence, by Shelah's singular compactness theorem for free groups \cite{Sh:52}, singular compactness also holds for the property
\(
\Ext(-,\mathbb{Z}) = 0.
\)

Throughout the paper, all groups are abelian. For standard terminology from set-theoretic algebra, we refer to Eklof–Mekler \cite{EM02} and G\"{o}bel–Trlifaj \cite{GT}. Background from group theory may be found in Fuchs \cite{fuchs}.
\medskip

\section {Preliminaries}

\medskip
\noindent
In this section we establish the notation and conventions used throughout the paper, and recall several basic facts from set-theoretic algebra.  All groups are abelian and written additively.

\medskip

For abelian groups \(G\) and \(H\), we write
\(
\Hom(G,H) := \Hom_{\mathbb{Z}}(G,H).
\)

\begin{notation}\label{nzu}
	Let $u$ be an index set, and define
	\(
	\mathbb{Z}_{[u]} := \bigoplus_{\alpha \in u} \mathbb{Z} x_\alpha,
	\)
	so that $\langle x_\alpha : \alpha \in u \rangle$ is a basis for $\mathbb{Z}_{[u]}$.
\end{notation}

\begin{definition}
\label{def1} Let  $\kappa$   be an infinite  cardinal and    ${G}$ be an abelian group.
\begin{itemize}
\item[(a)] The group   ${G}$ is called $\kappa$-free if every subgroup of ${G}$ of cardinality
less than $\kappa$ is free.

\item[(b)] The group  ${G}$ is said to be strongly $\kappa$-free if it is $\kappa$-free and in addition
every subset of ${G}$ of cardinality   $<\kappa$  is contained in a subgroup $H$ of ${G}$ of
cardinality  $<\kappa$  such that ${G}/H$ is $\kappa$-free.
\end{itemize}
 \end{definition}
\begin{remark}
 We note that
if $\kappa$ is  a singular cardinal, then   $G$ is $\kappa$-free if and only if $G$ is $\lambda$-free
for every regular cardinal $\lambda < \kappa.$
\end{remark}

  Recall that a ring \(R\) is said \emph{left-perfect} if every flat left \(R\)-module is projective. Since \(\mathbb{Q}\) is flat but not projective, it follows that \(\mathbb{Z}\) is not left-perfect. Moreover, the following result shows that for a weakly compact cardinal \(\kappa\), every \(\kappa\)-free group of cardinality \(\kappa\) is free.

  \begin{fact}\label{we}
  	(See \cite[Theorems VI.3.2 and VII.1.4]{EM02}.)
  	Let \(\lambda\) be a weakly compact cardinal and let $R$ be a  non-left-perfect ring. If \(M\) is a \(\leq\!\lambda\)-generated module which is \(\lambda\)-free, then \(M\) is free.
  \end{fact}

  We now verify that \( \Pr_{\lambda} \) holds whenever \(\lambda\) is weakly compact.

  \begin{fact}\label{weapr}
  	Assume \(\lambda\) is a weakly compact cardinal and \(G\) is a group of size \(\lambda\) such that
  	\(
  	\Hom(G',\mathbb{Z}) \neq 0
  	\)
  	for every nontrivial subgroup \(G' \subseteq G\) of cardinality \(<\lambda\). Then \(\Hom(G,\mathbb{Z}) \neq 0\). In particular, \( \Pr_{\lambda} \) holds.
  \end{fact}

  \begin{proof}
  	Let \(\theta > \lambda\) be sufficiently large and regular, and let \(M\) be an elementary submodel of \((\mathcal{H}(\theta),\in)\) of cardinality \(\lambda\) such that
  	\[
  	M^{<\lambda} \subseteq M
  	\qquad\text{and}\qquad
  	\lambda,\, G \in M.
  	\]
  	Since \(\lambda\) is weakly compact, there exists a model \(N\) of cardinality \(\lambda\) with \(N^{<\lambda} \subseteq N\) and an elementary embedding \(j : M \to N\) such that
  	\[
  	j,\, M \in N,\qquad
  	j(\alpha) = \alpha \text{ for all } \alpha < \lambda,\qquad
  	j(\lambda) > \lambda.
  	\]
  	
  	In \(N\), the group \(j(G)\) has cardinality \(j(\lambda)\), and \(G\) is a subgroup of \(j(G)\) of size \(\lambda < j(\lambda)\). By elementarity,
  	\[
  	N \models \text{``}\Hom(G,\mathbb{Z}) \neq 0\text{''}.
  	\]
  	This implies \(\Hom(G,\mathbb{Z}) \neq 0\) in the ambient universe, and the claim follows.
  \end{proof}

\begin{definition}\label{dia}
	Suppose $\lambda > \mu \geq \aleph_0$ are regular cardinals, and let $S \subseteq \lambda$ be stationary.
	
	\begin{enumerate}
		\item \emph{Jensen's diamond $\diamondsuit_\lambda(S)$} asserts the existence of a sequence
		$(S_\alpha \mid \alpha \in S)$ such that for every $X \subseteq \lambda$, the set
		\[
		\{\alpha \in S \mid X \cap \alpha = S_\alpha\}
		\]
		is stationary.
		
		\item A useful consequence of $\diamondsuit_\lambda(S)$ is the following.
		Let $A = \bigcup_{\alpha<\lambda} A_\alpha$ and $B = \bigcup_{\alpha<\lambda} B_\alpha$ be two $\lambda$-filtrations with $|A_\alpha|, |B_\alpha| < \lambda$.
		Then there exists a sequence of functions
		\(
		(g_\alpha : A_\alpha \to B_\alpha \mid \alpha < \lambda)
		\)
		such that, for any function $g: A \to B$, the set
		\(
		\{\alpha \in S \mid g\!\!\restriction_{A_\alpha} = g_\alpha\}
		\)
		is stationary in $\lambda$.
		
		\item $S$ is \emph{non-reflecting} if for every limit ordinal $\delta < \lambda$ of uncountable cofinality, the set $S \cap \delta$ is non-stationary in $\delta$.
		
		\item We set
		\(
		S^\lambda_\mu := \{\alpha < \lambda \mid \cf(\alpha) = \mu\}.
		\)
	\end{enumerate}
\end{definition}

\begin{definition}
\label{e22}
Let $\mathcal{K}$ be the class of all uncountable cardinals $\mu   $ so that if $K_1\subseteq \mathbb{Z}$ is a subgroup, then  there is $ (H_{K_1},\eta_{K_1})$ equipped with the following properties:

\begin{enumerate}
\item[$(\alpha)$]  $H_{ K_1}$ is an abelian group extending $ \mathbb{Z} _{[\mu ]}$,

\item[$(\beta)$]  $H_{ K_1}/ \mathbb{Z}_{[\mu ]}$ is $\mu $-free,

\item[$(\gamma)$]  $\eta_{ K_1} \in {}^{\mu }(K_1)$,

\item[$(\delta)$]  there is no homomorphism $f:H_{ K_1}\to K_1$
  such that    $f(x_\alpha) =
  \eta_{ K_1}(\alpha)$ for $ \alpha < \mu $:
	$$\xymatrix{
&\mathbb{Z}_{[\{\alpha\}]}\ar[r]^{\subseteq}\ar[d]_{ \tilde{\eta}}&\mathbb{Z}_{[\mu ]}\ar[r]^{\subseteq}&\ar[dll]^{\nexists f}H_{ K_1}\\
 & K_1
	&&&}$$
where $\tilde{\eta}(x_\alpha):=\eta_{ K_1}(\alpha)$.

\end{enumerate}
\end{definition}

We now give sufficient conditions for a cardinal $\mu$ being in $\mathcal{K}$.
We first recall the following result of Eklof--Mekler, see~\cite{EM77}.

\begin{fact}\label{ej}
	Let $\mu$ be an uncountable regular cardinal.
	Assume that $S\subseteq S^\mu_{\aleph_0}$ is stationary and non-reflecting,
	and that $\diamondsuit_S$ holds.
	Then there exists an indecomposable, strongly $\mu$-free abelian group $H$
	of cardinality $\mu$.
\end{fact}

\begin{lemma}\label{e31}
	We have  $\mu\in\mathcal{K}$
  provided that:
	\begin{enumerate}
		\item[(a)] $\mu=\cf(\mu)>\aleph_0$;
		\item[(b)] $S\subseteq S^\mu_{\aleph_0}$ is stationary and non-reflecting;
		\item[(c)] $\diamondsuit_S$ holds.
	\end{enumerate}
\end{lemma}

\begin{proof}
	By Fact~\ref{ej}, there exists an indecomposable, strongly $\mu$-free abelian
	group $H$ of cardinality $\mu$. We first note that
	$\Hom(H,\mathbb Z)=0$. Indeed, if $0\neq f\in\Hom(H,\mathbb Z)$, then
	$\Rang(f)\cong\mathbb Z$ is free, so the exact sequence
	\[
	0\longrightarrow\ker(f)\longrightarrow H\longrightarrow\Rang(f)\longrightarrow 0
	\]
	splits. Since $H$ is indecomposable, this forces $\ker(f)=0$, hence $H$ embeds
	into $\mathbb Z$, contradicting $|H|=\mu>\aleph_0$.
	As $H$ is strongly $\mu$-free, there exists a subgroup
	$K\subseteq H$ of cardinality $<\mu$ such that $\mathbb Z\subseteq K$ and
	$H/K$ is $\mu$-free. Without loss of generality, write
	$K=\bigoplus_{i<\theta}\mathbb Z$ for some $\theta<\mu$.
	Fix a partition $\mu=\bigcup_{i<\mu} I_i$ such that the $I_i$ are pairwise
	disjoint, $|I_i|=\theta$ for all $i<\mu$, and $I_i<I_j$ whenever $i<j$.
	Define
	\[
	G:=\bigoplus_{i<\mu} G_i,
	\]
	where each $G_i\cong H$ and $G_i/\mathbb Z_{I_i}$ is $\mu$-free.
	Since $H$ is torsion-free, we have $\mathbb Z\subseteq H$, and hence
	\[
	\mathbb Z_{[\mu]}\subseteq \bigoplus_{i<\mu} G_i = G.
	\]
	
	Moreover,
	\[
	\Hom(G,\mathbb Z)
	\cong \Hom\Bigl(\bigoplus_{i<\mu}G_i,\mathbb Z\Bigr)
	\cong \prod_{i<\mu}\Hom(H,\mathbb Z)=0.
	\]
	
	We next verify that $G$ is $\mu$-free. Let $L\leq G$ with $|L|<\mu$.
	There exists $I\subseteq\mu$ with $|I|<\mu$ and subgroups
	$L_i\leq G_i$, each of size $<\mu$, such that
	$L\subseteq\bigoplus_{i\in I} L_i$.
	Since each $L_i$ is free, so is $\bigoplus_{i\in I} L_i$, and hence $L$ is free.
	
	Finally, note that
	\(
	G/\mathbb Z_{[\mu]}\cong\bigoplus_{i<\mu} G_i/\mathbb Z_{I_i},
	\)
	and each summand on the right-hand side is $\mu$-free. Thus
	$G/\mathbb Z_{[\mu]}$ is $\mu$-free.
	Therefore  $\mu\in\mathcal{K}$.
	For any $K_1\subseteq\mathbb Z$, set
	$(H_{ K_1},\eta_{ ,K_1})=(G,\eta)$,
	where $\eta\colon\mu\to K_1$ is arbitrary. If $K_1\neq 0$, then
	$K_1\cong\mathbb Z$, and hence
	\(
	\Hom(G,K_1)=0,
	\)
	which verifies $(\delta)$ from Definition~\ref{e22}.
\end{proof}
We also need the following well-known result of Kurepa.

\begin{fact}\label{kurepa}
	Assume $\cf(\lambda)>\aleph_0$ and let $\mathcal{T}$ be a tree of height $\lambda$ whose levels are all finite. Then $\mathcal{T}$ has a cofinal branch.
\end{fact}
\medskip
\section {Controlling $\Hom(G,\bbZ)$}

In this section we prove our main result (see Theorem~\ref{e25}).
\begin{discussion}
	Recall that a cardinal $\kappa$ is \emph{measurable} if it is uncountable and there exists a non-principal $\kappa$-complete ultrafilter $\mathcal{D}$ on $\kappa$, meaning that for every subset $S \subseteq \mathcal{D}$ of cardinality less than $\kappa$, the intersection $\bigcap S$ belongs to $\mathcal{D}$. It is a classical result that the existence of measurable cardinals cannot be established within ZFC.
\end{discussion}

\begin{definition}
	Let $G$ be an abelian group. The \emph{dual} of $G$ is the abelian group $\Hom(G, \mathbb{Z})$, denoted by $G^\ast$. For $g \in G$, define $\psi_g : G^\ast \to \mathbb{Z}$ by evaluation
	\[
	\psi_g(f) := f(g), \quad f \in G^\ast.
	\]
	The assignment $g \mapsto \psi_g$ defines a canonical map $\psi : G \to G^{\ast\ast}$. We say that $G$ is \emph{reflexive} if $\psi$ is an isomorphism.
\end{definition}

\begin{fact}(L\"{o}s--Eda, Shelah; see \cite[Corollary III.1.5]{EM02} and \cite{sh:904}).\label{los}
	Let $\mu = \mu_{\mathrm{first}}$ be the first measurable cardinal. The following hold:
	\begin{enumerate}
		\item[(a)] For any $\theta < \mu$, the group $\mathbb{Z}^{(\theta)}$ is reflexive; in fact, its dual is $\mathbb{Z}^{\theta}$.
		\item[(b)] For any $\lambda \ge \mu$, the group $\mathbb{Z}^{(\lambda)}$ is not reflexive.
		\item[(c)] There exists a reflexive group $G \subset \mathbb{Z}^{\mu}$ of cardinality $\mu$.
	\end{enumerate}
\end{fact}

Let $\Pr$ be any property of abelian groups, and let $\lambda$ be a cardinal. Recall that \emph{compactness for $(\lambda, \Pr)$} means that if $G$ is a group of cardinality $\lambda$ and
\[
\text{``for all } G' \subseteq G \text{ with } |G'| < \lambda, \text{ $G'$ has $\Pr$''},
\]
then $G$ itself has $\Pr$. In this paper, we focus on the following specific property of abelian groups:

\begin{notation}
	For a cardinal $\lambda$, let $\Pr_\lambda$ denote the property:
If $G$ is a group of size $\lambda$, and if for any nontrivial subgroup $G' \subseteq G$ of size less than $\lambda,$ $\Hom(G',\bbZ) \ne 0$, then $\Hom(G,\mathbb{Z})\neq0$.

\end{notation}

We now turn to the primary framework for our construction.

\begin{definition}\label{e11}
Let $\theta$ be a cardinal.	\begin{enumerate}
		\item Let $\mathbf M_{1, \theta}$ be the class of objects
		\[
		\mathbf m = \big(\lambda_{\mathbf m}, \langle G^{\mathbf m}_\alpha : \alpha \le \alpha_{\mathbf m} \rangle, S_{\mathbf m} ,\langle f_{\mathbf m,s} : s \in S_{\mathbf m} \rangle \big)
		\]
		consisting of:
		\begin{enumerate}
			\item[$(a)$]
			\begin{enumerate}
				\item[$(\alpha)$] $\lambda_{\mathbf m} = \cf(\lambda_{\mathbf m}) > \aleph_0$,
				\item[$(\beta)$] $\lambda_{\mathbf m} \ge \alpha_{\mathbf m} := \lh(\mathbf m)$, where $\lh(\mathbf m)$ denotes the length of $\mathbf m$.
			\end{enumerate}
			
			\item[$(b)$]
			\begin{enumerate}
				\item[$(\alpha)$] $\langle G^{\mathbf m}_\alpha : \alpha \le \alpha_{\mathbf m} \rangle$ is an increasing and continuous sequence of abelian groups,
				\item[$(\beta)$] $|G^{\mathbf m}_\alpha| < \lambda_{\mathbf m}$ for $\alpha < \alpha_{\mathbf m}$.
			\end{enumerate}
			
			\item[$(c)$] $G^{\mathbf m}_\alpha / G^{\mathbf m}_0$ is free.
			
			\item[$(d)$] The set
			\[
			\{\beta < \alpha_{\mathbf m} : G^{\mathbf m}_{\beta +1} / G^{\mathbf m}_\beta \text{ is not free} \}
			\]
			is non-reflecting and stationary.
			
			\item[$(e)$]
			\begin{enumerate}
				\item[$(\alpha)$] $S_{\mathbf m}$ is a set of cardinality $\le \theta$,
				\item[$(\beta)$] $f_{\mathbf m,s} \in \Hom(G^{\mathbf m}_{\alpha_{\mathbf m}}, \mathbb{Z})$ for each $s \in S_{\mathbf m}$.
			\end{enumerate}
			
			\item[$(f)$] The family $\langle f_{\mathbf m,s} : s \in S_{\mathbf m} \rangle$ is a free basis of a subgroup of $\Hom(G^{\mathbf m}_{\alpha_{\mathbf m}}, \mathbb{Z})$.
		\end{enumerate}
		
		\item The class $\mathbf M_{2, \theta}$ is defined as above, with the modification that in item (a)$(\beta)$ we set $\alpha_{\mathbf m} = \lambda_{\mathbf m}$, and we further require:
		\begin{enumerate}
			\item[(g)] For every $f \in \Hom(G^{\mathbf m}_{\alpha_{\mathbf m}}, \mathbb{Z})$, there exists $h \in \Hom(\prod_{S_{\mathbf m}} \mathbb{Z}, \mathbb{Z})$ such that
			\[
			f(x) = h(\langle f_{\mathbf m,s}(x) : s \in S_{\mathbf m} \rangle) \quad \text{for all } x \in G^{\mathbf m}_{\alpha_{\mathbf m}}.
			\]
		Namely, 	$h$ is a factorization of $f$ through
			\[
			\pi\in \Hom(G^{\mathbf m}_{\alpha_{\mathbf m}}, \mathbb{Z}^{S_{\mathbf m}}),
			\]
			defined by
			\(
			\pi(x) := \big( f_{\mathbf m,s}(x) \mid s \in S_{\mathbf m} \big).
			\)
			\item[(h)] The homomorphism   $\pi$
			 is   surjective.
			
			\item[(i)] For any nonzero subgroup $G' \subseteq G^{\mathbf m}_\alpha$ with $\alpha < \alpha_{\mathbf m}$, we have
			\(
			\Hom(G', \mathbb{Z}) \neq 0.
			\)
		\end{enumerate}
	\end{enumerate}
\end{definition}

We are now in a position to state and prove our main result:
\begin{theorem}
\label{e25}
Assume  that:
\begin{enumerate}
\item[(a)]  $\langle \lambda_i:i < \kappa\rangle$ is an
  increasing sequence of regular cardinals with limit $\lambda$,
and $\lambda_i=\mu_i^+$, for some regular cardinal $\mu_i$,
\item[(b)]  $\aleph_0 < \kappa  \leq \chi < \lambda_0$ and
$\kappa, \chi$ are regular,

\item[(c)]  $S_i \subseteq
  S^{\lambda_i}_{\aleph_0}$ is stationary and non-reflecting,
  and $\diamondsuit_{S_i}$ holds,

\item[(d)] $T_i \subseteq
  S^{\mu_i}_{\aleph_0}$ is stationary and non-reflecting,
  and $\diamondsuit_{T_i}$ holds,

\item[(e)] there is no measurable cardinal $\le \lambda$.

\end{enumerate}

Then there is a  $\chi$-free  abelian group $G$ of cardinality
$\lambda$ which is counterexample to singular compactness in
$\lambda$ for $\Pr_\lambda$.
\end{theorem}

 \begin{remark}
The sets  $T_i$ and  $S_i$   have
	cardinality $\mu_i$ and $\lambda_i$, respectively. In particular, they are different, though both of them have concentrated on ordinals
of countable cofinality.
 \end{remark}

\begin{proof}
We construct a $\chi$-free abelian group $G$ of cardinality $\lambda$ endowed with the following property:
for any nontrivial subgroup $G' \subseteq G$ of smaller cardinality, we have
\[
\Hom(G',\mathbb{Z}) \neq 0,
\]
while
$
\Hom(G,\mathbb{Z}) = 0.
$

The idea of the construction is to equip the cardinals $\lambda_i$ with an underlying tree structure $\mathcal{T}_i$, which allows precise control of homomorphisms from the building blocks $G_i$ of the group
\[
G := \bigcup_{i<\kappa} G_i.
\]

The proof is organized in four stages.
Stage (A) defines the tree structure $\mathcal{T}_i$.
Stage (B), which is more involved, constructs by induction a system of $\chi$-free groups together with homomorphisms from them to $\mathbb{Z}$.
Stage (C) verifies that subgroups of smaller cardinality have nontrivial duals, and
Stage (D) shows that there are no non-zero homomorphisms from $G$ to $\mathbb{Z}$.

\medskip
\underline{Stage (A)}:  We define a tree $\mathcal{T}$ of height $\kappa$, whose $i$-th level $\mathcal{T}_i$ is defined as follows:

\begin{enumerate}
	\item[$(*)^i_A:$] $\mathcal{T}_i$ is the set of all sequences $\eta$ satisfying:
	\begin{enumerate}
		\item[(a)] $\eta$ has length $i+1$,
		
		\item[(b)] for each $j \le i$, $\eta(j) = (\eta(j,1), \eta(j,2))$,
		
		\item[(c)] for each $j \le i$, $\eta(j,1) < \lambda_j$ and $\eta(j,2) < \kappa$,
		
		\item[(d)] if $j_1 < j_2 \le i$, then $\eta(j_1,1) \le \eta(j_2,1)$ and $\eta(j_1,2) \le \eta(j_2,2)$,
		
		\item[(e)] the range of $\eta$ is finite,
		
		\item[(f)] if $j_1 < j_2 \le i$ and the sequence $\langle \eta(j,1) : j \in [j_1,j_2] \rangle$ is constant, then $j_2 < \eta(j_1,2)$.
	\end{enumerate}
\end{enumerate}

Set $\mathcal{T} := \bigcup_{i<\kappa} \mathcal{T}_i$, ordered by the end-extension relation $\triangleleft$. Then $(\mathcal{T}, \triangleleft)$ is a tree of height $\kappa$, whose $i$-th level is $\mathcal{T}_i$. Moreover, if $\eta \in \mathcal{T}_i$ and $i < j < \kappa$, there exists $\nu \in \mathcal{T}_j$ such that $\eta \triangleleft \nu$; that is, $\eta = \nu \restriction (i+1)$.

For truncated trees, we define, for $\alpha \le \lambda_i$,
\[
\mathcal{T}_{i,\alpha} := \{ \eta \in \mathcal{T}_i : \eta(i,1) \le \alpha \}.
\]
In particular, $\mathcal{T}_i = \mathcal{T}_{i,\lambda_i}$.

\begin{claim}\label{nobranch}
	The tree $(\mathcal{T},\triangleleft)$ has no branches of length $\kappa$.
\end{claim}
	\begin{PROOF}{\ref{nobranch}}
	Assume, towards a contradiction, that there exists a branch
	\[
	b := \langle \eta_i : i<\kappa \rangle
	\]
	of $\mathcal{T}$. Then $\langle \eta_i : i<\kappa \rangle$ is $\triangleleft$-increasing. It follows that the sequence
	\[
	\langle \eta_i(i,1) : i<\kappa \rangle
	\]
	is non-decreasing in the ordinals.
	By clause $(*)^i_{A}(e)$, every initial segment takes only finitely many values, and since $\kappa = \cf(\kappa) > \aleph_0$, the sequence must eventually stabilize. That is, there exists $i_* < \kappa$ such that
	\[
	\eta_i(i,1) = \eta_{i_*}(i_*,1) \quad \text{for all } i \in [i_*,\kappa).
	\]
	
	On the one hand, by $(*)^i_{A}(f)$, we have
	\[
	\eta(i_*,2) > i \quad \text{for all } i<\kappa,
	\]
	while on the other hand, $\eta(i_*,2) < \kappa$. This is a contradiction.
\end{PROOF}
\noindent
\underline{Stage (B):}
We define $\mathbf m_i$ by induction on $i < \kappa$ such that:

\begin{enumerate}
	\item[$(*)^i_B:$] The sequence $\mathbf m_i$ satisfies:
	\begin{enumerate}
		\item[(a)] $\mathbf m_i = \bigl(\lambda_{\mathbf m_i}, \langle G^{\mathbf m_i}_\alpha : \alpha \le \alpha_{\mathbf m_i}\rangle, S_{\mathbf m_i},  \langle f_{\mathbf m_i,s} : s \in S_{\mathbf m_i} \rangle\bigr) \in \mathbf M_{1, \lambda_i}$,
		
		\item[(b)] $\lambda_{\mathbf m_i} = \alpha_{\mathbf m_i} = \lambda_i$, and the set of elements of $G^{\mathbf m_i}_{\lambda_i}$ has cardinality $\lambda_i$,
		
		\item[(c)] $G_{< i} := \bigcup\{G^{\mathbf m_j}_{\lambda_j} : j < i\} \cup \{0\}$,
		
		\item[(d)] $G^{\mathbf m_i}_0 := G_{< i}$,
		
		\item[(e)] $S_{\mathbf m_i} := \mathcal{T}_i = \mathcal{T}_{i, \lambda_i}$,
		
		\item[(f)] if $j < i$, then $\mathbf m_j \le \mathbf m_i$, i.e.,
		\[
		\eta \in \mathcal{T}_j \wedge \nu \in \mathcal{T}_i \wedge \eta \triangleleft \nu \implies f_{\mathbf m_j,\eta} \subseteq f_{\mathbf m_i,\nu},
		\]
		which can be depicted as:
		\[
		\xymatrix{
			& 0 \ar[r] & G^{\mathbf m_j}_{\lambda_j} \ar[r]^{\subseteq} \ar[d]_{f_{\mathbf m_j,\eta}} & G^{\mathbf m_i}_{\lambda_i} \ar[dl]^{f_{\mathbf m_i,\nu}} \\
			& & \mathbb{Z} &&
		}
		\]
		
		\item[(g)] $\langle f_{\mathbf m_i,\eta} : \eta \in \mathcal{T}_i \rangle$ is an independent subset of $\Hom(G^{\mathbf m_i}_{\lambda_i}, \mathbb{Z})$,
		
		\item[(h)] $\bigcap\{\Ker(f_{\mathbf m_i,\eta}) : \eta \in \mathcal{T}_i\} = \{0\}$,
		
		\item[(i)] for any $f \in \Hom(G^{\mathbf m_i}_{\lambda_i}, \mathbb{Z})$, there exist $\alpha < \lambda_i$ and $h \in \Hom\bigl(\prod_{\eta \in \mathcal{T}_{i,\alpha}} \mathbb{Z}, \mathbb{Z}\bigr)$ such that
	$h$ is a factorization of $f$ through
	\(
	\pi_i\in \Hom(G^{\mathbf m_i}_{\alpha_{{\mathbf m_i}}}, \mathbb{Z}^{S_{{\mathbf m_i}}}),
	\)
	defined by
	\(
	\pi_i(x) = \big( f_{\mathbf m_i,s}(x) \mid s \in S_{\mathbf m_i} \big).
	\)
	\end{enumerate}
\end{enumerate}
\begin{remark}\label{rems}
	For each $\alpha < \lambda_i$, the cardinality of $\mathcal{T}_{i,\alpha}$ is less than $\lambda_i$.
	This fact will be useful to show that
	\[
	|G^{\mathbf{m}_j}_\alpha| < \lambda_{\mathbf{m}_j} \quad \text{for } \alpha < \lambda_{\mathbf{m}_j},
	\]
	see the subsequent discussion in $(\sharp)$ below.
\end{remark}

Assume that $i < \kappa$ and that the sequence $\langle \mathbf{m}_j : j < i \rangle$ has been defined.
Fix a diamond sequence
\[
\langle F_{i,\delta} : \delta \in S_i \rangle, \quad F_{i,\delta} : \delta \to \mathbb{Z}.
\] \begin{notation}
 	Let $\langle \beta_i(\gamma) : \gamma < \lambda_i \rangle$ be an increasing, continuous sequence of ordinals cofinal in $\lambda_i$, with $\beta_i(0) = 0$.
 \end{notation}

 We proceed by setting
 $G_{<i} := \bigcup_{j<i} G^{\mathbf{m}_j}_{\lambda_j} \cup \{0\}$. Also,
for $\eta \in \mathcal{T}_i$,  define $f_{<i,\eta}: G_{<i} \to \mathbb{Z}$ by
 	\[
 	f_{<i,\eta} := \bigcup_{j<i} f_{\mathbf{m}_j, \eta \restriction (j+1)}.
 	\]
In particular, $G_{<0} = \{0\}$ and $f_{<0,\eta}: G_{<0} \to \mathbb{Z}$ is the zero map.
We shall choose $\mathbf{m}_{i,\gamma}$ by induction on $\gamma < \lambda_i$ such that:

\begin{enumerate}
\item[$(*)^\gamma_{C}:$]
\begin{enumerate}
\item[(a)] $\mathbf m_{i,0}$ is defined as
 \begin{enumerate}
\item[$(\alpha)$] $\lh(\mathbf
  m_{i,0}) =0$,
 \item[$(\beta)$]  $\lambda_{\mathbf m_{i,0}}:=\chi+\sup_{j<i}\lambda_{\mathbf m_{j}}$  with the convention that $\sup_{j<0}\lambda_{\mathbf m_{j}}=0$
\item[$(\gamma)$]  $G^{\mathbf m_{i,0}}_0:= G_{< i}$,
\item[$(\delta)$] $S_{\mathbf m_{i,0}}:=\cT_{i,\beta_i(0)}$,
\item[$(\epsilon)$] for $\eta \in \cT_{i,\beta_i(0)}$, $f_{\mathbf m_{i, 0}, \eta}:=f_{<i, \eta}$.
\end{enumerate}
\item[(b)]  $\langle \mathbf m_{i,\gamma}: \gamma < \lambda_i \rangle$ is an increasing
 and continuous sequence from $\mathbf M_{1, \lambda_i}$ with
 $S_{ \mathbf m_{i,\gamma}}=\cT_{i,  \beta_i(\gamma)}$ and $\lh(\mathbf
  m_{i,\gamma}) = \alpha_{i,\gamma} < \lambda_i$, which means:
 \begin{enumerate}
 \item[$(\alpha)$]  if $\rho < \gamma$, then $\mathbf m_{i,\rho} \le
\mathbf m_{i,\gamma}$,

 \item[$(\beta)$] if $\gamma$ is a limit ordinal, then $\mathbf m_{i,\gamma}:=\bigcup_{\rho<\gamma} \mathbf m_{i,\rho}$. In particular, we have the following equalities:
 \begin{enumerate}
 \item[$(\beta_1)$] $\alpha_{i,\gamma}=\sup_{\rho<\gamma} \alpha_{i,\rho}$,

\item[$(\beta_2)$]   $G^{\mathbf m_{i,\gamma}}_{ \alpha_{i,\gamma}}=\bigcup_{\rho < \gamma}G^{\mathbf m_{i,\rho}}_{ \alpha_{i,\rho}}$,

\item[$(\beta_3)$] $S_{\mathbf m_{i,\gamma}}=\cT_{i,\beta_i(\gamma)}$,
\item[$(\beta_4)$]
     $f_{\mathbf m_{i,\gamma},\eta }=f_{i, <\eta} \cup \bigcup_{\rho < \gamma} f_{\mathbf m_{i,\rho}, \eta \rest \rho+1}$ for any  $\eta \in \cT_{i, \beta_i(\gamma)}$.
\end{enumerate}
\end{enumerate}

\item[(c)] If $\rho < \gamma$, and $\rho \notin S_i$, then
  $G^{\mathbf m_{i,\gamma}}_{\alpha_{i, \gamma}}/G^{\mathbf m_{i,\rho}}_{\alpha_{i, \rho}}$ is free.

\item[(d)]  $\bigcap\{\Ker(f_{\mathbf m_{i,\gamma},\eta}): \eta \in \cT_{i, \beta_i(\gamma)}\} = \{0\}$.

\item[(e)]  $G^{\mathbf m_{i,\gamma}}_{\alpha_{i, \gamma}}$ has  set of
  elements an  ordinal $\delta_i(\gamma)< \lambda_i$.

\item[(f)] Recall that $\langle F_{i,\delta} : \delta \in S_i \rangle$ denotes the diamond sequence.
We now collect the following notations and assumptions:

\begin{itemize}
	\item[$(\alpha)$] $\gamma = \alpha_{i,\gamma} \in S_i$,
	
	\item[$(\beta)$] The underlying set of $G^{\mathbf{m}_{i,\gamma}}_{\alpha_{i,\gamma}}$ is $\gamma$,
	
	\item[$(\gamma)$] $\Rang(F_{i,\gamma}) \subseteq \mathbb{Z}$ is nonzero; in particular, $\Rang(F_{i,\gamma}) = n \mathbb{Z} \cong \mathbb{Z}$ for some nonzero $n \in \mathbb{Z}$,
	
	\item[$(\delta)$] $F_{i,\gamma}$ is a homomorphism from $G^{\mathbf{m}_{i,\gamma}}_{\alpha_{i,\gamma}}$ onto $\Rang(F_{i,\gamma})$,
	
	\item[$(\epsilon)$] $F_{i,\gamma} \notin \langle f_{\mathbf{m}_{i,\gamma},\eta} : \eta \in \mathcal{T}_{i, \beta_i(\gamma)} \rangle$,
	where $\langle f_{\mathbf{m}_{i,\gamma},\eta} : \eta \in \mathcal{T}_{i, \beta_i(\gamma)} \rangle$ denotes the subgroup of $\Hom(G^{\mathbf{m}_{i,\gamma}}_{\alpha_{i,\gamma}}, \mathbb{Z})$ generated by $\{ f_{\mathbf{m}_{i,\gamma},\eta} : \eta \in \mathcal{T}_{i, \beta_i(\gamma)} \}$.
\end{itemize}

We then define $\mathbf{m}_{i,\gamma+1}$ so that $F_{i,\gamma}$ has no extension to a homomorphism from $G^{\mathbf{m}_{i,\gamma+1}}_{\alpha_{i,\gamma+1}}$ into $K_1$. Namely, we have the following commutative diagram:
  $$\xymatrix{
  	 &G^{\mathbf m_{i,\gamma}}_{\alpha_{i, \gamma}}\ar[r]^{\subseteq}\ar[d]_{F_{i,\gamma}}&G^{\mathbf m_{i,\gamma+1}}_{\alpha_{i, \gamma+1}}\ar[d]^{\nexists}\\
  	 & \Rang(F_{i,\gamma})\ar[r]^{=}&\Rang(F_{i,\gamma}),
  	&&&}$$
\end{enumerate}
\end{enumerate}where the right vertical arrow does not exist.
To simplify notation, we let $G_{i, \gamma} := G^{\mathbf{m}{i,\gamma}}{\alpha_{i,\gamma}}$. We then define $G_{i, \gamma, \rho} := G_{i, \gamma}$ for any $\rho \le \alpha_{i,\gamma}$.
Also, we abbreviate the function $f_{\mathbf{m}{i,\gamma},\eta}$ as $f_{i, \gamma, \eta}$ for any index $\eta \in \mathcal{T}_{i, \beta_i(\gamma)}$.

The starting case $\gamma = 0$ is trivial, since it can be defined as in $(*)^\gamma_C$(a).
By the induction hypothesis and the definition of $f_{i,0,\eta}$, we have:
\begin{itemize}
\item[$\circ$] The sequence $\langle f_{i,0,\eta} : \eta \in \mathcal{T}_{i, \beta_i(0)} \rangle$ is an independent subset of $\Hom(G_{i,0}, \mathbb{Z})$,
\item[$\circ$] Each $f_{i,0,\eta}$ extends $f_{<i,\eta}$,
\item[$\circ$] $\bigcap \{ \Ker(f_{i,0,\eta}) : \eta \in \mathcal{T}_{i, \beta_i(0)} \} = \{0\}$.
\end{itemize}

If $\gamma$ is a limit ordinal, set $\alpha_{i,\gamma} = \sup_{\rho < \gamma} \alpha_{i,\rho}$ and define $\mathbf{m}_{i,\gamma}$ as in $(*)^\gamma_C$(b)$(\beta)$.

\medskip
 Suppose that $\mathbf{m}_{i,\gamma}$ has already been defined.
We now proceed to define $\mathbf{m}_{i,\gamma+1}$. First, assume that one of the following cases occurs:

\begin{itemize}
	\item[$(h_1)$] $\gamma \notin S_i$ or at least one of the hypotheses $(*)^\gamma_C$(f)$(\alpha)$-$(\delta)$ fails, or
	\item[$(h_2)$] $\gamma \in S_i$, the hypotheses $(*)^\gamma_C$(i)$(\alpha)$-$(\delta)$ hold, but either
	\begin{itemize}
		\item $G_{i,\gamma}$ does not have domain $\gamma$, or
		\item $F_{i,\gamma} \notin \Hom(G_{i,\gamma}, \mathbb{Z})$, or
		\item $F_{i,\gamma} \in \langle f_{i,\gamma,\eta} : \eta \in \mathcal{T}_{i, \beta_i(\gamma)} \rangle$.
	\end{itemize}
\end{itemize}

In this case, define $\mathbf{m}_{i,\gamma+1}$ as follows:
\begin{enumerate}
	\item $\alpha_{i,\gamma+1} := \alpha_{i,\gamma}+1$,
	\item $\mathbf{m}_{i,\gamma} \leq \mathbf{m}_{i,\gamma+1}$,
	\item $S_{\mathbf{m}_{i,\gamma+1}} := \mathcal{T}_{i, \beta_i(\gamma+1)}$,
	\item $G_{i,\gamma+1} := G_{i, \gamma+1, \alpha_{i,\gamma}+1} := G_{i,\gamma} \oplus \mathbb{Z}_{[u_{i,\gamma}]}$, where $u_{i,\gamma} := \mathcal{T}_{i, \beta_i(\gamma+1)}$,
	\item For $\eta \in \mathcal{T}_{i, \beta_i(\gamma+1)}$, set $f_{i,\gamma+1,\eta} := f_{i,\gamma,\eta} \oplus \pi_\eta$, where $\pi_\eta : \mathbb{Z}_{[u_{i,\gamma}]} \to \mathbb{Z} x_\eta$ is the projection, and for $\eta \notin \mathcal{T}_{i, \beta_i(\gamma)}$, $f_{i,\gamma,\eta}$ is the zero map.
\end{enumerate}

\medskip
\noindent
Then suppose that $\mathbf{m}_{i,\gamma}$ is defined and the following holds:
\begin{itemize}
	\item[$(h_3)$] $\gamma \in S_i$, $G_{i,\gamma}$ has domain $\gamma$, hypotheses $(*)^\gamma_C$(f)$(\alpha)$-$(\delta)$ hold, $F_{i,\gamma} \in \Hom(G_{i,\gamma},\mathbb{Z})$, and $F_{i,\gamma} \notin \langle f_{i,\gamma,\eta} : \eta \in \mathcal{T}_{i, \beta_i(\gamma)} \rangle$.
\end{itemize}
In this case, we define $\mathbf{m}_{i,\gamma+1}$ so as to satisfy $(*)^\gamma_B$(i). Let $\alpha_{i,\gamma+1} := \alpha_{i,\gamma}+1$ and set
\[
G^{\mathbf{m}_{i,\gamma+1}}_\rho := G^{\mathbf{m}_{i,\gamma}}_\rho = G_{i,\gamma,\rho}, \quad \forall \rho \le \alpha_{i,\gamma}.
\]
It remains to define
\[
G_{i,\gamma+1} = G^{\mathbf{m}_{i,\gamma+1}}_{\alpha_{i,\gamma}+1} \quad \text{and} \quad f_{i,\gamma+1,\eta} = f_{\mathbf{m}_{i,\gamma+1},\eta} : G_{i,\gamma+1} \to \mathbb{Z}, \quad \eta \in \mathcal{T}_{i, \beta_i(\gamma+1)}.
\]
For each $\beta < \lambda_i$, define
\[
G^{[\beta]}_{i,\gamma} := \{ x \in G_{i,\gamma} : \eta \in \mathcal{T}_i \wedge \eta(i,1) < \beta \implies f_{i,\gamma,\eta}(x) = 0 \}.
\]
Then $\langle G^{[\beta]}_{i,\gamma} : \beta < \lambda_i \rangle$ is increasing, and since $|G_{i,\gamma}| < \lambda_i$, there exists $\beta_{i,\gamma} < \lambda_i$ such that
\[
G^{[\beta]}_{i,\gamma} = G^{[\beta_{i,\gamma}]}_{i,\gamma}, \quad \forall \beta \in (\beta_{i,\gamma}, \lambda_i).
\]
Set
\[
K_{i,\gamma}
:= \Rang\!\bigl(F_{i,\gamma}\restriction G^{[\beta_{i,\gamma}]}_{i,\gamma}\bigr),
\]
and note that
\[
\chi + \sum_{j<i}\lambda_j
+ \bigl|\mathcal{T}_{i,\beta_i(\gamma+1)}\bigr|
+ \aleph_0
< \lambda_i.
\]
Consequently,
\[
\chi + \sum_{j<i}\lambda_j
+ \bigl|\mathcal{T}_{i,\beta_i(\gamma+1)}\bigr|
+ \aleph_0
\leq \mu_i.
\]
We set $\mu_{i,\gamma}:=\mu_i$.
Combining Lemma~\ref{e31} with assumption~\ref{e25}(d), we obtain
$\mu_{i,\gamma}\in\mathcal{K}$. Using Notation~\ref{nzu}, define
\[
(K_{i,\gamma})_{[u]}
:= \bigoplus_{\alpha\in u} K_{i,\gamma}x_\alpha
\quad\text{for any index set }u.
\]
Also, we define
\(
(H_*,\phi_*)
:= \bigl(H_{K_{i,\gamma}},
\phi_{K_{i,\gamma}}\bigr),
\)
so that
\begin{itemize}
	\item[$\circ$] $H_*$ is a free abelian group of size $\mu_{i,\gamma}$ extending $(K_{i,\gamma})_{[u_{i,\gamma}]}$,
	\item[$\circ$] $\phi_* : u_{i,\gamma} \to K_{i,\gamma} \setminus \{0\}$,
	\item[$\circ$] $H_*/(K_{i,\gamma})_{[\mu_{i,\gamma}]}$ is $\mu_{i,\gamma}$-free,
	\item[$\circ$] there is no homomorphism $f : H_* \to K_{i,\gamma}$ such that $f(x_\eta) = \phi_*(\eta)$ for all $\eta \in \mathcal{T}_{i, \beta_i(\gamma+1)}$.
\end{itemize}

\medskip
Finally, for each $\beta$ with $\beta_{i,\gamma} \le \beta < \lambda_i$ and $b \in K_{i,\gamma} = \Rang(F_{i,\gamma} \restriction G^{[\beta_{i,\gamma}]}_{i,\gamma})$, there exists $y_{b,\beta} \in G^{[\beta_{i,\gamma}]}_{i,\gamma} \subseteq G_{i,\gamma}$ such that
\begin{enumerate}
	\item[$(*)_{\beta, b}$]
	\begin{enumerate}
		\item If $\eta \in \cT_i$ and $\eta(i,1) < \beta$, then $f_{i,\gamma,\eta}(y_{b,\beta}) = 0$,
		\item $F_{i,\gamma}(y_{b,\beta}) = b$.
	\end{enumerate}
\end{enumerate}

\medskip
Since $|G_{i,\gamma}| < \lambda_i$, for each $b \in K_{i,\gamma}$ as above, there exists some fixed $y_b \in G_{i,\gamma}$ such that the set
\[
X_b := \{\beta < \lambda_i : \beta_{i,\gamma} \le \beta \text{ and } y_{b,\beta} = y_b \}
\]
is stationary in $\lambda_i$.

The assignment $x_\eta \mapsto y_{\phi_*(\eta)}$ induces a morphism
\[
g_{i,\gamma} : (K_{i,\gamma})_{[u_{i,\gamma}]} \longrightarrow G_{i,\gamma}.
\]
Recall that $u_{i,\gamma} = \mathcal{T}_{i,\beta_i(\gamma+1)}$ and that
\[
\id : (K_{i,\gamma})_{[u_{i,\gamma}]} \longrightarrow H_{\mathbf k_{i,\gamma}, K_{i,\gamma}}
\]
denotes the natural inclusion.

\begin{itemize}
	\item[$(\sharp)$] Define
	\[
	G_{i,\gamma+1} := \frac{G_{i,\gamma} \times H_{\mathbf k_{i,\gamma},K_{i,\gamma}}}
	{\big\langle (g_{i,\gamma}(k), -\id(k)) : k \in (K_{i,\gamma})_{[u_{i,\gamma}]} \big\rangle}.
	\]
\end{itemize}

Recall that $H_{\mathbf k_{i,\gamma},K_{i,\gamma}}$ is of size
$
\mu_{i,\gamma}  < \lambda_i.
$
Combining this with an inductive argument along with $(\ast)$, we conclude that $|G_{i,\gamma+1}| < \lambda_i$.

\begin{notation}
	For $g \in G_{i,\gamma}$ and $h \in H_{\mathbf k_{i,\gamma},K_{i,\gamma}}$, let $[(g,h)] \in G_{i,\gamma+1}$ denote the equivalence class of $(g,h)$.
\end{notation}

By the definition of the push-out construction, this gives us two maps
\[
h_{i,\gamma}: G_{i,\gamma} \longrightarrow G_{i,\gamma+1}, \qquad
k_{i,\gamma}: H_{\mathbf k_{i,\gamma}, K_{i,\gamma}} \longrightarrow G_{i,\gamma+1},
\]
such that
$
h_{i,\gamma} \circ g_{i,\gamma} = k_{i,\gamma}.
$ In fact $h_{i,\gamma}$ is defined by
$
h_{i,\gamma}(x) := [(x,0)].
$
The situation is depicted in the following commutative diagram:
\[
\xymatrix{
	& H_{\mathbf k_{i, \gamma}, K_{i, \gamma}} \ar[rr]^{k_{i, \gamma}} && G_{i, \gamma+1} \\
	& (K_{i, \gamma})_{[u_{i, \gamma}]} \ar[rr]^{g_{i, \gamma}} \ar[u]^{\id} && G_{i, \gamma} \ar[u]_{h_{i, \gamma}}
}
\]

\begin{claim}
	\label{ntrick}
	Let $g \in G_{i,\gamma}$ and $h \in (K_{i,\gamma})_{[u_{i,\gamma}]}$. Then there exists an element
	$\tilde{g} \in G_{i,\gamma}$ such that
	$
	[(g,h)] = [(\tilde{g},0)]
	$
	in $G_{i,\gamma+1}$.
\end{claim}

\begin{PROOF}{\ref{ntrick}}
	Set
	$
	\tilde{g} := g + g_{i,\gamma}(h) \in G_{i,\gamma}.
	$
	Then
	\[
	(g,h) - (\tilde{g},0)
	= (g - \tilde{g},\, h)
	= (-g_{i,\gamma}(h),\, h)
	\in
	\Big\langle \big( g_{i,\gamma}(k), -\id(k) \big) : k \in (K_{i,\gamma})_{[u_{i,\gamma}]} \Big\rangle.
	\]
	Hence $[(g,h)] = [(\tilde{g},0)]$ in $G_{i,\gamma+1}$.
\end{PROOF}

Clearly, the map $h_{i,\gamma}: G_{i,\gamma} \to G_{i,\gamma+1}$ is an embedding, and  we may identify
$
G_{i,\gamma} \subseteq G_{i,\gamma+1}
$
via $h_{i,\gamma}$.
Let $g \in G_{i,\gamma}$ and $h \in H_{\mathbf k_{i,\gamma},K_{i,\gamma}}$, and define
\[
\psi : G_{i,\gamma+1} \longrightarrow
H_{\mathbf k_{i,\gamma},\, K_{i,\gamma}} / (K_{i,\gamma})_{[u_{i,\gamma}]},
\quad
\psi([(g,h)]) := h + (K_{i,\gamma})_{[u_{i,\gamma}]}.
\]
As consequences of well-known general
facts, $\psi$ is well-defined, surjective, and also
\[
\ker(\psi)
= \{[(g,h)] : g \in G_{i,\gamma},\, h \in (K_{i,\gamma})_{[u_{i,\gamma}]} \}
\stackrel{(\ref{ntrick})}{=}
\{[(\tilde{g},0)] : g \in G_{i,\gamma}\}
\cong G_{i,\gamma}.
\]
So,
\(
G_{i,\gamma+1}/G_{i,\gamma}
\;\cong\;
H_{\mathbf k_{i,\gamma},\, K_{i,\gamma}} / (K_{i,\gamma})_{[u_{i,\gamma}]}.
\)
Since the latter quotient is $\mu_{i,\gamma}$-free by assumption, the same holds for
$G_{i,\gamma+1}/G_{i,\gamma}$.
The next key observation is the non-extendability property of $F_{i,\gamma}$. Suppose, toward a contradiction, that
$F: G_{i,\gamma+1} \to K_{i,\gamma}$ extends $F_{i,\gamma}$. Then the map
\[
f := F \circ k_{i,\gamma} : H_{\mathbf k_{i,\gamma}, K_{i,\gamma}} \longrightarrow K_{i,\gamma}
\]
satisfies
\[
f(x_\eta) = F \circ k_{i,\gamma}(x_\eta) = F_{i,\gamma} \circ g_{i,\gamma}(x_\eta) = \phi_*(\eta)
\]
for all $\eta \in u_{i,\gamma}$, contradicting the choice of $(H_{\mathbf k_{i,\gamma},K_{i,\gamma}}, \phi_*)$.

\medskip
\noindent
Next, we define the map $f_{i,\gamma+1,\eta}$. Let $\eta \in \cT_{i,\beta_i(\gamma+1)}$. For any
$h \in H_{\mathbf k_{i,\gamma},K_{i,\gamma}}$ and $g \in G_{i,\gamma}$, set
\[
f_{i,\gamma+1,\eta}([(h,g)]) := f_{i,\gamma,\eta}(g).
\]
This defines a homomorphism
\[
f_{i,\gamma+1,\eta} = f_{\mathbf m_{i,\gamma+1},\eta} : G_{i,\gamma+1} \longrightarrow \mathbb{Z}.
\]
To verify that $f_{i,\gamma+1,\eta}$ is well-defined, it suffices to check that
\(
f_{i,\gamma,\eta} \circ g_{i,\gamma} = 0.
\)
Indeed, for a given $\eta \in \cT_{i,\beta_i(\gamma+1)}$, choose $\beta \in X_{\phi_*(\eta)}$ with
$\eta(i,1) < \beta$. Then by $(*)_{\beta,\phi_*(\eta)}$(a),
\[
f_{i,\gamma,\eta} \circ g_{i,\gamma}(x_\eta)
= f_{i,\gamma,\eta}(y_{\phi_*(\eta)})
= f_{i,\gamma,\eta}(y_{\phi_*(\eta),\beta})
= 0.
\]

\medskip
\noindent
Also, the family
\(
\{f_{i,\gamma+1,\eta} : \eta \in \cT_{i,\beta_i(\gamma+1)}\}
\)
is independent, and
\[
\bigcap_{\eta \in \cT_{i,\beta_i(\gamma+1)}} \Ker(f_{i,\gamma+1,\eta}) = \{0\}.
\]
\medskip
\noindent
Finally, for $\eta \in \cT_i$, define
\[
f_{i, \eta} := \bigcup_{\gamma < \lambda_i} f_{i, \gamma, \eta}.
\]
We are now prepared to state the following claim.

\begin{claim}\label{claim2}
	The set $\{ f_{i,\eta} : \eta \in \mathcal{T}_i \}$ generates
	$\Hom(G^{\mathbf{m}_i}_{\lambda_i}, \mathbb{Z})$.
\end{claim}

\begin{PROOF}{\ref{claim2}}
	Suppose, toward a contradiction, that there exists
	\[
	f \in \Hom(G^{\mathbf{m}_i}_{\lambda_i}, \mathbb{Z}) \setminus \langle f_{i,\eta} : \eta \in \mathcal{T}_i \rangle.
	\]
	Choose $\gamma \in S$ such that $G_{i,\gamma}$ has domain $\gamma$, $f \restriction \gamma = F_{i,\gamma}$, and
	\[
	f \restriction \gamma \notin \langle f_{i,\gamma,\eta} : \eta \in \mathcal{T}_i \rangle.
	\]
By construction, the map $f\restriction \gamma$ cannot be extended to a homomorphism
from $G_{i,\gamma+1}$ to $\mathbb{Z}$, yielding a contradiction.
	\end{PROOF}

We also note that clause $(\ast)^\gamma_B$(i) holds by Claim~\ref{claim2}. Indeed, given any
\(
f \in \Hom(G^{\mathbf{m}_i}_{\lambda_i}, \mathbb{Z}),
\)
there exist $\eta_0, \dots, \eta_{n-1} \in \mathcal{T}_i$ and coefficients
$\alpha_0, \dots, \alpha_{n-1} \in \mathbb{Z}$ such that
\(
f = \sum_{k=0}^{n-1} \alpha_k \, f_{i,\eta_k}.
\)
Pick $\alpha < \lambda_i$ large enough so that
$\eta_k(i,1) \leq \alpha$ for all $k<n$.
By Fact~\ref{los}(a), for each $i<\kappa$ we have
\[
\Hom\!\left(\prod_{\cT_{i,\alpha}} \mathbb{Z},\, \mathbb{Z}\right)
\cong \bigoplus_{\eta \in \cT_{i,\alpha}} \Hom(\mathbb{Z},\mathbb{Z})
\cong \bigoplus_{\eta \in \cT_{i,\alpha}} \mathbb{Z}x_\eta.
\]
In particular, there exists
\( h \in \Hom(\prod_{\cT_{i,\alpha}}\mathbb{Z}, \mathbb{Z}) \)
such that
\[
h\big(\langle f_{i,\eta} : \eta \in \cT_i \rangle\big)
= \sum_{k<n} \alpha_k f_{i,\eta_k}.
\]

\medskip
\noindent
Clause $(\ast)^\gamma_C$(b) follows from the fact that $S_i$ is non-reflecting. Hence there exists a club
$C \subseteq \gamma$ with $\min(C)=\rho$ such that
$C \cap S_i = \emptyset$. Then
\( G_{i,\gamma}/G_{i,\rho} \)
is the union of the increasing and continuous sequence
\(
\langle G_{i,\tau}/G_{i,\rho} : \tau \in C \rangle
\).
By the induction hypothesis, each
\( G_{i,\tau}/G_{i,\mu} \)
is free for all $\mu<\tau$ in $C$, and therefore
$G_{i,\gamma}/G_{i,\rho}$ is free as well.

\medskip
Having defined the sequence $\langle \mathbf{m}_{i,\gamma} : \gamma < \lambda_i \rangle$ with properties $(*)^\gamma_C$, we set
\[
\mathbf{m}_i := \bigcup_{\gamma < \lambda_i} \mathbf{m}_{i,\gamma}.
\]
This completes the inductive construction of $\langle \mathbf{m}_i : i < \kappa \rangle$ as required by $(*)^i_B$.

\medskip
\noindent

\underline{Stage (C):} In this step, we show that for each $i$,
$\mathbf{m}_i \in \mathbf{M}_{2,\lambda_i}$ (see Definition~\ref{e11}(2)).

Items (a)--(e) of Definition~\ref{e11}(1) and the equalities
$\alpha_{\mathbf{m}_i} = \lambda_i = \lambda_{\mathbf{m}_i}$ are immediate.

\medskip
\noindent
For clause (i), let $\alpha < \lambda_i$ and $0 \neq G' \subseteq G^{\mathbf{m}_i}_\alpha$.
Then there exists $\gamma < \lambda_i$ such that
$G' \subseteq G^{\mathbf{m}_{i,\gamma}}_{\alpha_{i,\gamma}}$.
Fix $0 \neq x \in G'$. By property $(*)^\gamma_C$(d),
\[
\bigcap \{ \Ker(f_{i,s}) : s \in \mathcal{T}_i \} = \{0\}.
\]
Hence there exists $s \in \mathcal{T}_i$ with $f_{i,s}(x) \neq 0$.
In particular, the restriction $f_{i,s}\restriction G' \in \Hom(G',\mathbb{Z})$ is non-zero,
as required.

\medskip
\noindent
\underline{Stage (D)}: In this stage we conclude the proof of Theorem \ref{e25}.
For each $i<\kappa$, set
\[
G_i := G_{\mathbf{m}_i, \lambda_i},
\qquad
G_{<i} := \bigcup_{j<i} G_j .
\]
Then $\langle G_i : i<\kappa \rangle$ is increasing and continuous, and for each $i$ the quotient
$G_i / G_{<i}$
is $(\chi + \sum_{j<i} \lambda_j)$-free.
Indeed, for all $\gamma < \lambda_i$ we have $\mu_{i,\gamma} \ge \chi + \sum_{j<i}\lambda_j$, and each $G_{i,\gamma+1}/G_{i,\gamma}$ is $\mu_{i,\gamma}$-free.

Define
\[
G := \bigcup_{i<\kappa} G_i.
\]
Then $G$ is an abelian group of size $\lambda$.
We first show that $G$ is $\chi$-free.
To this end, let $H \leq G$ be a subgroup of size $<\chi$.
Then the sequence $$\langle H \cap G_i : i<\kappa \rangle$$ is increasing and continuous.
For each $i<\kappa$,
\[
(H \cap G_i)/(H \cap G_{<i})
\cong
((H \cap G_i)+G_{<i})/G_{<i},
\]
and this group is free since $G_i/G_{<i}$ is $(\chi+\sum_{j<i}\lambda_j)$-free, and therefore $\chi$-free.
Hence
$H = \bigcup_{i<\kappa}(H \cap G_i)$
is free.

\medskip

Next, let $H \le G$ be a nonzero subgroup of size $<\lambda$.
Choose $i<\kappa$ such that $H \cap G_{<i} \neq \{0\}$ and $|H|<\lambda_i$.
By Definition~\ref{e11}(2)(i),
\[
\Hom(H \cap G_i,\mathbb Z)\neq 0.
\]
Moreover,
\[
(H \cap G_i)/(H \cap G_{<i})
\cong
((H \cap G_i)+G_{<i})/G_{<i}
\]
is free, and therefore $\Hom(H,\mathbb Z)\neq 0$.

\medskip

Finally, we show that $\Hom(G,\mathbb Z)=0$.
Suppose, towards a contradiction, that
$0 \neq f \in \Hom(G,\mathbb Z)$.
By $(\ast)^i_B$(i), for each $i<\kappa$ there exist $\alpha_i<\lambda_i$ and
\[
h_i \in \Hom\Bigl( \prod_{\mathcal{T}_{i,\alpha_i}} \mathbb{Z}, \mathbb{Z} \Bigr)
\]
such that
\[
f(x) = h_i\bigl( \langle f_{\mathbf{m}_i,\eta}(x) : \eta\in\mathcal{T}_{i,\alpha_i}\rangle \bigr),
\qquad x\in G_i.
\]
By Fact~\ref{los}(a),
\[
\Hom\Bigl(\prod_{\mathcal{T}_{i,\alpha_i}}\mathbb Z,\mathbb Z\Bigr)
\cong
\bigoplus_{\eta\in\mathcal{T}_{i,\alpha_i}} \mathbb Z x_\eta.
\]

By \cite[Corollary III.3.3]{EM02}, for each $i$ there exists a finite set $u_i\subseteq\mathcal{T}_{i,\alpha_i}$ such that
\[
f_{\mathbf{m}_i,\eta}(x)=0 \ \text{for all}\ \eta\in u_i \ \Longrightarrow\ f(x)=0,
\qquad x\in G_i.
\]
Since $\kappa=\cf(\kappa)>\aleph_0$, there exists $n_*$ such that the set
\[
\mathcal{V}_1 := \{ i<\kappa : |u_i| = n_* \}
\]
is unbounded in $\kappa$.

For $i<j<\kappa$, define
\[
\pr_{i,j} : \mathcal{T}_j \to \mathcal{T}_i,
\qquad
\pr_{i,j}(\eta) = \eta\restriction(i+1).
\]
If $\eta\in u_j$, then $\eta\restriction(i+1)\in u_i$, since $G_i \subseteq G_j$ and
$f_{\mathbf{m}_i,\,\eta\restriction(i+1)} \subseteq f_{\mathbf{m}_j,\eta}$.
Thus $\pr_{i,j}$ maps $u_j$ onto $u_i$.

Let $\mathcal{T}$ be the tree of height $\kappa$ whose $i$-th level is $u_i$.
This is a well-defined tree of uncountable height with finite levels.
By Fact~\ref{kurepa}, $\mathcal{T}$ carries a cofinal branch, contradicting Claim~\ref{nobranch}.
\end{proof}

\section*{Acknowledgements}
The authors are grateful to the referees for a very careful reading of the paper and for
helpful comments that have improved both its clarity and presentation.

\end{document}